\documentclass[9pt,shortpaper,twoside,web]{ieeecolor}
\usepackage{times} 
\usepackage{epsfig,graphicx,generic,booktabs,subfigure,cite,textcomp}
\usepackage{pgf,algpseudocode,algorithm,caption,bm}
\usepackage{amsmath,amssymb,mathrsfs,amsfonts}
\newtheorem{lemma}{Lemma}
\newtheorem{theorem}{Theorem}

\newtheorem{remark}{Remark}
\newtheorem{definition}{Definition}
\newtheorem{example}{Example}

\newtheorem{assumption}{Assumption}

\title{Achieving Social Optimum in Non-convex Cooperative Aggregative Games: A Distributed Stochastic Annealing Approach}
\author{Yinghui~Wang, Xiaoxue~Geng $^{*}$\thanks{$*$ Corresponding Author}, Guanpu Chen, and Wenxiao Zhao
\thanks{Y. Wang is with  School of Automation and Electrical Engineering, University of Science
and Technology Beijing, Beijing 100083, China e-mail: (wangyinghuisdu@163.com).}
\thanks{X. Geng, G. Chen and W. Zhao are with Key Laboratory of Systems and Control, Academy of Mathematics and Systems Science, Chinese Academy of Sciences, Beijing 100190, China and also with School of Mathematical Sciences, University of Chinese Academy of Sciences, Beijing 100049, China. (email: gengxiaoxue17@mails.ucas.ac.cn; chengp@amss.ac.cn; wxzhao@amss.ac.cn)}
\thanks{This research was supported by NSFC (72101026, 61621063) and the State Key Laboratory
of Intelligent Control and Decision of Complex Systems.}
}

\begin{document}
	
	\maketitle
	\thispagestyle{empty}
	\pagestyle{empty}
	
\begin{abstract}
This paper designs a distributed stochastic annealing algorithm for non-convex cooperative aggregative games, whose agents' cost functions not only depend on agents' own decision variables but also rely on the sum of agents' decision variables. To seek the the social optimum of cooperative aggregative games, a distributed stochastic annealing algorithm is proposed, where the local cost functions are non-convex and the communication topology between agents is time varying. The weak convergence to the social optimum of the algorithm is further analyzed. A numerical example is given to illustrate the effectiveness of the proposed algorithm.
\end{abstract}

\begin{IEEEkeywords}
cooperative aggregative game, social optimum, distributed stochastic algorithm, non-convex, annealing.
\end{IEEEkeywords}

\section{Introduction}
In the past decades, distributed games over networks has received great attention because of their wide range of applications in smart grids, public environments, and communication network\cite{barrera2014dynamic, cornes2016aggregative, ye2016game, yi2019operator, chen2021distributed, li2017off, zhang2018game}. Various kinds of networked games, including resource allocation games, aggregative games, two-network zero-sum games, mean-field games  and so on, have been studied. Most of these aforementioned distributed works  investigated noncooprate mechanism and explored Nash equilibria, where no agent can improve his own revenue  through unilaterally changing his strategy and maximize agents’ own revenue over the network.

However, the seeking of Nash equilibria can not maximize the interests of games over the whole network. In fact, social optimum, origned from \cite{green1961social} and of which Pareto optimum is a useful necessary condition \cite{heydenreich2007games}, seeks global optimum of the whole network. Still, for intrinsic interest or means of measuring the efficiency of different Nash equilibria, the social optimum has been widely studied in various situations, including resource allocation games \cite{johari2004efficiency, gkatzelis2016optimal, von2013optimal}, mean-field games \cite{nourian2012nash, wang2017social, li2019connections} and so on. Different from Nash equilibria,  the social optimum offers  cooperation mechanism for networked games. When 
seeking the social optimum of games, the decision made by each agent needs to consider the win-win cooperation over the network rather than the maximization of individual interests. 

Aggregative games have  attracted extensive attention among  cooperative game models.  Actually, coopererative aggregative games have multitude of practical applications and various examples,  e.g.,  task assignment problems \cite{marden2013distributed}, drivers allocation over transportation networks \cite{brown2017optimal} and demand side management in smart grids \cite{chen2014autonomous}.  Although sharing the same aggregative funtions, coopererative aggregative games seek the social optimum while traditionally non-coopererative aggregative games seek Nash equilibria. Compared with seeking Nash equilibria, agents are more blind and random in making decisions for seeking social optimum in coopererative aggregative games, which brings chanllenges to both the designs and analysis of distributed algorithms in cooperative aggregative games.  On one hand,  traditional distributed algorithms\cite{koshal2016distributed, deng2018distributed, belgioioso2020distributed}  seeking Nash equilibria of non-coopererative aggregative games are no longer suitable in cooperative situations.  On the other hand, the rarely designed distributed  gradient-tracking work  \cite{li2021distributed},  studied the linear convergence rates whose agents'  cost functions are  strong convex. However, the algorithm proposed in \cite{li2021distributed} cannot deal with more complex non-convex, constrained, and stochastic cooperative aggregative games. 
 
Moreover, we are interested in solving stochastic distributed cooperative aggregative games, because stochasticity play important roles in the study of distributed convex problems. Actually, many stochastic algorithms have been designed for solving distributed problems. The earlier studied case \cite{sundhar2010distributed, yuan2020stochastic, wang2019distributed} is to minimize distributed unconstrainted and constrainted problems, whose global function is separable and composed of local convex and strongly convex functions. The more complex case is to study Nash equilibriua of uncoopererative games, whose agents' local functions depend on agents' own decision variables but also other agents' decision variables.  For example, in \cite{lei2020synchronous}, algorithms were designed for the best-response schemes of uncooperative  stochastic games while in \cite{meigs2019learning, shokri2020leader} were designed for uncooperative aggregative games. The aforementioned stochastic algorithms were mainly designed for solving problems with uncertain function information or communication topology between agents. 

Meanwhile, stochastic metheds are  also efficient for solving non-convex problems. Distributed stochastic  gradient algorithms (DSGD)
\cite{tatarenko2017non, vlaski2021distributed} were proposed for seeking local optima of non-convex problems. Still, although the more complex game situations have not been studied, stochastic annealing  algorithms \cite{swenson2019annealing} were designed to find the global solution (corresponding to the social optimum in games) of distributed unconstrainted problems. Compared with distributed stochastic  gradient algorithms (DSGD), the additional greedy factors given in distributed stochastic annealing algorithms \cite{swenson2019annealing} established escape from local optima in probability.

The above facts motivated this paper to study social optimum of non-convex cooperative aggregative games. Challenges mainly come from the more complex non-convex coopereative game setting, for which  \cite{meigs2019learning, shokri2020leader}  is not suitable. As mentioned before, stochastic annealing algorithms \cite{gelfand1991recursive, swenson2019annealing} are efficient to deal with non-convex functions. Therefore, we design a distributed stochastic annealing algorithm to deal with non-convex functions, followed by the convergence analysis. The contributions of this paper are summarized as follows. 

\begin{itemize}
\item[(a)]
We consider the seeking of the social optimum for non-convex cooperative aggregative games. Compared with the existing works\cite{koshal2016distributed, deng2018distributed, belgioioso2020distributed} designed for non-convex aggregative games, we seek the social optimum rather than Nash equilibria in games.  Compared with  \cite{li2021distributed},  we deal with the more complex non-convex and  stochastic aggregative games. The study of the social optimum in this paper extends the applications of stochastic algorithms in cooperative games.

\item[(b)]For non-convex cooperative aggregative games, we design a distributed annealing algorithm to seek social optimum. With Ito integral and  martingale theory, we show the weak convergence of the proposed algorithm and provide an electric vehicles example  to support the effectiveness of the proposed algorithm. The weak convergence of the proposed algorithm extends the applications of annealing algorithms \cite{gelfand1991recursive, swenson2019annealing} in distributed game settings. 
\end{itemize}

The rest of this paper is organized as follows. Preliminaries and our problem formulation are given in Section \ref{sec2}. In Section \ref{sec3}, a distributed annealing algorithm is introduced. The proposed algorithm is further analyzed in Section \ref{sec4} and a numerical example is presented in Section \ref{sec5}. Finally, the conclusion of this paper is given in Section \ref{sec6}.
\section{Preliminaries}\label{sec2}

In this section, mathematical preliminaries about probability theory and graph theory are first introduced. Then, the social optimum for aggregative games is formulated.

\subsection{Preliminaries}
Denote $(\Omega,\mathcal{F}, \mathbb{P})$ as the basic probability space, where $\Omega$ is the whole event space, $\mathcal{F}$ is the $\sigma$-algebra on $\Omega$, and $\mathbb{P}$ is the probability measure on $(\Omega,\mathcal{F})$.  Define $\mathcal{H}^{k}$ as a sequence of sub-$\sigma$-algebra on  $\mathcal{F}$.  Next, we give definitions of weak convergence in probability theory \cite{durrett2010probability}.

\begin{definition}
$X_{k}\Rightarrow X$ ($X_{k}$ weakly converges to $X$) if for any bounded continuous function $f$, we have $\mathbb{E}f\big(X_{k}\big)\rightarrow \mathbb{E}f\big(X\big)$.
\end{definition}

Then, a lemma of the convergence of nonnegative adapted process, which is useful in the proof of algorithm, is given.

\begin{lemma}\cite[Lemma 4.3]{kar2013distributed}\label{Lem:martingale}
Let $\big\{ \bm{z^{k}}\big\}$ be a nonnegative $\big\{\mathcal{H}^{k}\big\}$ adapted process and
\begin{align}\label{eq:martingale}
\bm{z^{k+1}}\leqslant \big(1-p^{k}\big)\bm{z^{k}}+q^{k}V^{k}(1+ O^{k}).
\end{align}
In \eqref{eq:martingale}, $\big\{ p^{k}\big\}$ is adaptive to $\big\{\mathcal{H}^{k+1}\big\}$ such that for all $k$, $p^{k}$ satisfies $0\leqslant p^{k}\leqslant 1$ and
\begin{align*}
\frac{a_{1}}{(k+1)^{\delta_{1}}}\leqslant \mathbb{E}\Big[p^{k}| \mathcal{H}^{k} \Big]\leqslant 1.
\end{align*}
 $\big\{ q^{k}\big\}$ satisfies $q^{k}\leqslant \frac{a_{2}}{(k+1)^{\delta_{2}}}$ with $a_{2},\delta_{2}>0$. Further, let $\big\{ V^{k}\big\}$ and $\big\{ O^{k}\big\}$ be $\mathcal{R}_{+}$ valued and adapted to $\big\{\mathcal{H}^{k+1}\big\}$. $\sup_{k\geqslant 0}\big\|V^{k}\big\|<\infty$ a.s. $\big\{ O^{k}\big\}$ is i.i.d. and independent of $\big\{\mathcal{H}^{k}\big\}$ with $\mathbb{E}\Big[ \big\|O^{k}\big\|^{2+\epsilon_{1}}\Big]<\kappa<\infty$ for some $\epsilon_{1}>0$ and constant $\kappa>0$. Then, for every $\delta_{0}$ such that
 \begin{align*}
 0\leqslant \delta_{0}<\delta_{2}-\delta_{1}-\frac{1}{2+\epsilon_{1}},
 \end{align*}
 we have
 \begin{align*}
 \lim_{k\rightarrow\infty }(k+1)^{\delta_{0}}\bm{z}^{k}=0,\quad  \mbox{a.s.}.
 \end{align*}
\end{lemma}

The communication topology between agents is modeled by a sequence of undirected time-varying networks $\mathcal{G}^{k}=(\mathcal{N},
\mathcal{E}^{k}),~k\geqslant 1$,  where $\mathcal{N} = \{1,2,...,n\}$ is the agent set, $k$ is the time index,  $\mathcal{E}^{k}\subset \mathcal{N}\times \mathcal{N} $ is the edge set at time $k$ which represents the link structure among agents. $W^{k}=[w_{ij}^{k}]_{i,j=1,\cdots,n}$ is the adjacency matrix, which describes the information communication protocol of $\mathcal{E}^{k}$, where $w_{ij}^{k}$ denotes the $ij$-th entry of matrix $W^{k}$.
\begin{align*}
w_{ij}^{k}=
\begin{cases}
1, \quad \mbox{if} \quad (i,j)\in\mathcal{E}^{k}\\
0,\quad \mbox{otherwise.}
\end{cases}
\end{align*}
The neighbours of agent $i$  at time $k$ is denoted by
\begin{align*}
N_{i}^{k}=\big\{ j\in\mathcal{N} |(j,i)\in \mathcal{E}^{k}\big\}.
\end{align*}
Agent $i$ has degree $d_{i}^{k}=|N_{i}^{k}|$ at time $k$. Define the degree matrix at time $k$ as the diagonal matrix $D^{k}=diag \big(d_{1}^{k},d_{2}^{k},\ldots,d_{n}^{k}\big)$. Define the Laplacian matrix at time $k$ as the positive semidefinite matrix $L^{k}=D^{k}-W^{k}$. The eigenvalues of Laplacian matrix $L^{k}$ can be ordered as $0=\lambda_{1}\big( L^{k}\big)\leqslant \lambda_{2}\big( L^{k}\big)\leqslant \ldots \leqslant \lambda_{n}\big( L^{k}\big)$, where the eigenvector corresponding to $\lambda_{1}\big( L^{k}\big)$ being $\Big(\frac{1}{\sqrt{n}}\Big)\bm{1}_{n}$. For any connected graph, we have $\lambda_{2}(L)>0$. The following assumption holds for the communication topology between agents.
\begin{assumption}\label{Ass:graph}\quad
\begin{itemize}
\item[(a)]The Laplacian matrices $L^{k}$ are independent and identically distributed (i.i.d.), with $L^{k}$ being independent of $\mathcal{H}^{k}$.
\item[(b)]The Laplacian matrices $L^{k}$  are connected in expectation, i,e. $\lambda_{2}(\bar{L})>0$ where $\bar{L}=\mathbb{E}(L^{k})$.
\end{itemize}
\end{assumption}
It should be noted that most communication topologies given in distributed algorithm designs, including connected graphs \cite{lu2020distributed}, time-varying strongly connected graphs \cite{wang2019distributed}, random networks\cite{kar2011convergence},  satisfy Assumption \ref{Ass:graph}.

\subsection{Problem Formulation}

Consider a set of $n$ agents indexed by $\{1, 2,\ldots, n\}$. The $i$-th player has a cost function $g_{i}(x_{i}, \bar{x})$, which depends on player $i$'s decision $x_i$ and the aggregate $\bar{x}$ of all players' decisions, i.e., $\bar{x}:=\frac{x_{1}+\ldots +x_{n}}{n}$. In an aggregative game, agent $i$ faces the following problem:
\begin{align}\label{eq:Problem}
\min_{x_{i}^{k}\in\mathcal{R}^{d}} g_{i}(x_{i},\bar{x})
\end{align}
 In problem \eqref{eq:Problem}, each agent $i$ only has privately access to  local  function $g_{i}$. Also, at time $k$, agent $i$ only aware the decision variables of its neighbours  $x_{j}$'s in $N_{i}^{k}$. In an cooperative aggregative model, the social optimum of problem \eqref{eq:Problem}, whose definition is given as follows, is investigated.
\begin{definition}
Define $\bm{x}=(x_1,\ldots, x_{n})^{\top}$, $\bm{x}_{-i}=(x_{-1},\ldots, x_{-n})^{\top}$ and $U(\bm{x}^{*},\bm{x}_{-i})=\sum_{i=1}^{n}U_{i}(x_{i},x_{-i})$. An strategy $\bm{x}^{*}=(x_{1}^{*},\ldots, x_{n}^{*})^{\top}$ is called a
social optimum if for all agents $i\in\mathcal{N}$,
$$U(\bm{x}^{*},\bm{x}_{-i}^{*})\leqslant U(\bm{x},\bm{x}_{-i}).$$
\end{definition}
\begin{remark}
In cooperative games, the study of the social optimum\cite{orbell1993social} focus on the global interest over the network, which 
helps rule makers (often the government) of games further improve the resource allocation. Social optimum also plays impotant roles in non-cooperative games. When there are several Nash equilibria in games, the seeking of social optimum is necessary to judge different Nash equilibria \cite{roughgarden2010algorithmic}.
\end{remark}
\begin{example} Consider a 2-agents game, where
\begin{align*}
\begin{cases}
f_{1}(x)=(x_{1}-2)^{2}+\frac{(x_{1}+x_{2})^{2}}{2},\\
f_{2}(x)=(x_{2}-3)^{2}+\frac{(x_{1}+x_{2})^{2}}{2}.
\end{cases}
\end{align*}
The Nash equilibrium for the 2-agents aggregative game is $\big(\frac{3}{4}, \frac{7}{4}\big)$ with $f(\bm{x}^{NE})=\frac{75}{8}$, while the social optimum of Problem \eqref{eq:Problem} is $\big(\frac{1}{3}, \frac{4}{3}\big)$ with  $f(\bm{x}^{SO})=\frac{75}{9}$.  The social optimum of Problem \eqref{eq:Problem} for the 2-agents aggregative game is more efficient than the  Nash equilibrium for the whole network.
\end{example}

The following assumption holds for local objective functions given in Problem \eqref{eq:Problem}.
\begin{assumption}\label{Ass:function}\quad
\begin{itemize}
\item[(a)]The partial gradient functions $\nabla_{1} g_{i}(x,y)$ and  $\nabla_{2} g_{i}(x,y)$ are Lipschitz continuous  with respect to $x$, i.e. there exists $L>0$ such that
\begin{align}\label{Ass:function 1}
\big\|\nabla_{1} g_{i}(x_{1},y)-\nabla_{1} g_{i}(x_{2},y)\big\|\leqslant L\big\|x_{1}-x_{2}\big\|,\notag\\
\big\|\nabla_{2} g_{i}(x_{1},y)-\nabla_{2} g_{i}(x_{2},y)\big\|\leqslant L\big\|x_{1}-x_{2}\big\|.
\end{align}
\item[(b)]The partial gradient functions $\nabla_{1} g_{i}(x,y)$ and  $\nabla_{2} g_{i}(x,y)$ are Lipschitz continuous  with respect to $y$, i.e. there exists $L>0$ such that
\begin{align}\label{Ass:function 2}
\big\|\nabla_{1} g_{i}(x,y_{1})-\nabla_{1} g_{i}(x,y_{2})\big\|\leqslant L\big\|y_{1}-y_{2}\big\|,\notag\\
\big\|\nabla_{2} g_{i}(x,y_{1})-\nabla_{2} g_{i}(x,y_{2})\big\|\leqslant L\big\|y_{1}-y_{2}\big\|.
\end{align}
\item[(c)]The partial gradient functions $\nabla_{1} g_{i}(x,y)$ and  $\nabla_{2} g_{i}(x,y)$ satisfy the following bounded gradient-dissimilarity condition:
\begin{align}\label{Ass:function 3}
\sup_{x,y\in\mathcal{R}^{d}}\big\|\nabla_{1} g_{i}(x,y)-\nabla_{1} G(x,y)\big\|<\infty,\; \forall i,\notag\\
\sup_{x,y\in\mathcal{R}^{d}}\big\|\nabla_{2} g_{i}(x,y)-\nabla_{2} G(x,y)\big\|<\infty,\; \forall i,
\end{align}
where $G(x,y)$ is a common function.
\end{itemize}
\end{assumption}

Different from the strong convexity assumption with respect to $x$ given in \cite{li2021distributed}, we do not require the convexity of local cost functions $g_{i}$s in this paper. Assumption \ref{Ass:function} requires  the boundedness of partial gradient functions $\nabla_{1} g_{i}(x,y)$ and  $\nabla_{2} g_{i}(x,y)$, which is realistic and common  in rule making for games.

\section{Algorithm Design}\label{sec3}
In this section, we design a  distributed annealing algorithm to seek the social optimum of  aggregative games.

\subsection{Algorithm Design}
The distributed annealing algorithm is presented in Algorithm \ref{Alg1}.
\begin{algorithm}[h]
\flushleft
\caption{ \bf Distributed annealing algorithm }
\label{Alg1}
\hspace*{0.02in}
{\bf Initialize:} $x_{i}(0)\in X$ for all $i=1,2,\ldots n$, stepsize sequences $\big\{\alpha^{k}\big\}$, $\big\{\beta^{k}\big\}$, $\big\{\gamma^{k}\big\}$, and noises sequences $\big\{\iota_{i}^{k}\big\}$, $\varsigma_{i}^{k} $.
\begin{algorithmic}[1]
\For{$k=0,\ldots T$}
\begin{align}
s_{i}^{k}\quad=&v_{i}^{k}-\beta^{k}\Big[\sum_{j=1}^{n}a_{ij}^{k}\big(v_{i}^{k}-v_{j}^{k}\big)\Big],\label{Alg1-1}\\
x_{i}^{k+1}=&x_{i}^{k}-\alpha^{k}\Big[ \nabla_{x_{i}^{k}} g_{i}\big(x_{i}^{k},s_{i}^{k}\big)+\varsigma_{i}^{k} \Big]+\gamma^{k}\iota_{i}^{k},\label{Alg1-2}\\
v_{i}^{k+1}=&s_{i}^{k}+x_{i}^{k+1}-x_{i}^{k},\label{Alg1-3}
\end{align}
  \EndFor
\end{algorithmic}
where $\nabla_{x_{i}^{k}} g_{i}\big(x_{i}^{k},s_{i}^{k}\big)=\nabla_{1} g_{i}\big(x_{i}^{k},s_{i}^{k}\big)+\frac{1}{n}\nabla_{2} g_{i}\big(x_{i}^{k},s_{i}^{k}).$
\end{algorithm}

In Algorithm \ref{Alg1}, $s_{i}^{k}$ is introduced for agent
$i$ to track $\bar{x}$, which is unavailable for agent $i$. Our algorithm is an ``annealing'' algorithm which is inspired by (distributed) annealing algorithm \cite{gelfand1991recursive, swenson2019annealing} for unconstrained optimization.  The noises $\varsigma_{i}^{k}$s allow partial gradients $\nabla_{x} g_{i}\big(x_{i}^{k},s_{i}^{k}\big)$ to be inexact, while  random variables $\iota_{i}^{k}$s are the cooperative factors in games. 
Still, with the distributed setting where agents could not share their stategies with the whole network,  the cooperative factors $\iota_{i}^{k}$s are randomly given. With Algorithm \ref{Alg1}, agents could cooperatively iterate to the social optimum of cooperative aggregative games.
\begin{remark}
A closely related work is the deterministic
algorithm in \cite{koshal2016distributed} for Nash equilibriums seeking of non-cooperative aggregative  games. For comparison, we write the  algorithm in \cite{koshal2016distributed} (referring as DAAG) here.
\begin{align}\label{Alg2-1}
\begin{cases}
\hat{s}_{i}^{k}\quad=&\hat{v}_{i}^{k}-\beta^{k}\Big[\sum_{j=1}^{n}a_{ij}^{k}\big(\hat{v}_{i}^{k}-\hat{v}_{j}^{k}\big)\Big],\\
\hat{x}_{i}^{k+1}=&\hat{x}_{i}^{k}-\alpha^{k}\nabla_{1} g_{i}\big(\hat{x}_{i}^{k},\hat{s}_{i}^{k}\big),\\
\hat{v}_{i}^{k+1}=&\hat{s}_{i}^{k}+\hat{x}_{i}^{k+1}-\hat{x}_{i}^{k}.
\end{cases}
\end{align}
Different from the  DAAG using $\nabla_{1} g_{i}\big(\hat{x}_{i}^{k},\hat{s}_{i}^{k}\big)$ in iterations of $\hat{x}_{i}^{k+1}$, Algorithm \ref{Alg1} make use of partial function $\nabla_{x} g_{i}\big(x_{i}^{k},s_{i}^{k}\big)$ with respect to $x$. Besides, noises $\varsigma_{i}^{k}$ and cooperative factors $\iota_{i}^{k}$ are also introduced to find social optimum for cooperative aggregative games.
\end{remark}

In addition, the step-size sequences $\big\{\alpha^{k}\big\}$, $\big\{\beta^{k}\big\}$ and $\big\{\gamma^{k}\big\}$ satisfy
\begin{align*}
\alpha^{k}=\frac{c_{\alpha}}{k}, \quad \beta^{k}=\frac{c_{\beta}}{k^{\tau_{\beta}}}, \;\mbox{and}\; \gamma^{k}=\frac{c_{\gamma}}{k^\frac{1}{2}\sqrt{\log\log k}}
\end{align*}
for large $k$, where $c_{\alpha}$, $c_{\beta}$ and $c_{\gamma}>0$ with $\tau_{\beta}\in \big(0, \frac{1}{2}\big)$.

The following conditions hold for  noises in Algorithm \ref{Alg1}.
\begin{itemize}
\item[(a)]Define ${\bm{\varsigma^k}}={\rm{col}}(\varsigma_1^k,\cdots,\varsigma_n^k)$ and ${\bm{\iota^k}}={\rm{col}}(\iota_1^k,\cdots,\iota_n^k)$. The sequence $\big\{\bm{\varsigma^{k}}\big\}$ is adapted to $\mathcal{H}^{k}=\sigma\big(\bm{x^{0}},\bm{v^{0}}, L^{0},\cdots, L^{k-1},\\\bm{\varsigma^{0}},\cdots,\bm{\varsigma^{k-1}},\bm{ \iota^{0}},\cdots, \bm{ \iota^{k-1}}\big)$, the $\sigma$-algebra corresponding to update process \eqref{Alg1-1}-\eqref{Alg1-3}, and there exists a  constant $C>0$ such that
\begin{align*}
\mathbb{E}\big[\bm{\varsigma^{k}} \big| \mathcal{H}^{k}\big]=0,
\end{align*}
and
\begin{align*}
\mathbb{E}\big[\big\| \bm{\varsigma^{k}}\big\|^{2} \big| \mathcal{H}^{k}\big]<C,
\end{align*}
for all $k\geqslant 0$.
\item[(b)]The sequence $\big\{\iota_{i}^{k}\big\}$is an i.i.d. sequence of $d$-dimensional following Gaussian distribution $N(\bm{0}, \bm{{I}_{d}})$. Further,  $\iota_{i}^{k}$ and $\iota_{j}^{k}$ are mutually independent for $i\neq j$.
\end{itemize}
The following lemma holds for gradient noise $\bm{\varsigma^{k}}$ in Algorithm \ref{Alg1}.
\begin{lemma}\label{Lem: gradient noise}
 For any $\delta>0$,  $\lim_{k\rightarrow \infty}\big(k+1\big)^{-\frac{1}{2}-\delta}\big\| \bm{\varsigma^{k}}\big\|=0$ holds for the gradient noise $\bm{\varsigma^{k}}$ in Algorithm \ref{Alg1} almost surely.
\end{lemma}
\begin{proof}
By Condition (a), for all $\epsilon>0$,
\begin{align}
\mathbb{P}\bigg( \big(k+1\big)^{-\frac{1}{2}-\delta}\big\| \bm{\varsigma^{k}}\big\|>\epsilon\bigg)\leqslant &\frac{1}{\epsilon^{2}\big(k+1\big)^{1+2\delta}}\mathbb{E}\Big[\big\| \bm{\varsigma}\big\|^{2}\Big]\notag\\\leqslant &\frac{C}{\epsilon^{2}\big(k+1\big)^{1+2\delta}}.
\end{align}
Since $\delta>0$, $\sum_{k=0}^{\infty}\frac{C}{\epsilon^{2}(k+1)^{1+2\delta}}<\infty$. Then by the Borel-Cantelli Lemma,
\begin{align}
\mathbb{P}\Big(\big(k+1\big)^{-\frac{1}{2}-\delta}\big\| \bm{\varsigma^{k}}\big\|>\epsilon \;\mbox{infinitely often}\Big)=0,
\end{align}
which yields the result.
\end{proof}

\section{Main Results}\label{sec4}
In this section, we first provide that variables $s_{i}^{k}$s  track  the network-averaged process $\bar{x}^{k}$ almost surely. Then, the weak convergence to the social optimum of distributed annealing algorithm is presented.

\subsection{Almost sure convergence of $s_{i}^{k}$}
Define by $\Lambda$ the consensus subspace in $\mathcal{R}^{nd}$,
\begin{align}\label{eq:consensus}
\Lambda=\big\{\bm{z}\in\mathcal{R}^{nd}: \bm{z}=\bm{1}_{n}\otimes a \;\mbox{for some\;} a \in \mathcal{R}^{d} \big\},
\end{align}
with $\Lambda^{\perp}$ as its orthogonal subspace in $\mathcal{R}^{nd}$. The following lemma holds with the communication topology given in Assumption \ref{Ass:graph}.

\begin{lemma}\cite[Lemma 4.3]{kar2013distributed}\label{Lem:graph}
Let $\{\bm{z}^{k}\}$ be an $\mathcal{R}^{nd}$-valued $\{\mathcal{H}^{k}\}$ adapted process such that $\bm{z}^{k}\in \Lambda^{\perp}$ for all $k$. With Assumption \ref{Ass:graph},  there exists a measurable $\{\mathcal{H}^{k+1}\}$ adapted $\mathcal{R}_{+}$ valued
process $\{r^{k}\}$ (depending on $\{\bm{z}^{k}\}$ and $\{L^{k}\}$) and a constant $c_r > 0$, such that $0\leqslant r^{k} \leqslant 1 $ holds a.s. and
\begin{align*}
\big\|\big(I_{nd}-\beta^{k} L^{k}\otimes I_{d}\big)\bm{z}^{k} \big\|\leqslant (1-r^{k})\big\|\bm{z}^{k}\big\|
\end{align*}
with
\begin{align*}
\mathbb{E}\big [r^{k} | \mathcal{H}^{k}\big] \geqslant \frac{c_{r}}{(k+1)^{\tau_{\beta}}}, \mbox{a.s.}
\end{align*}
for a sufficiently large $k$.
\end{lemma}
Next, variables $s_{i}^{k}$s track the network-averaged process $\bar{x}^{k}$ almost surely is given.
\begin{theorem}(Convergence to Consensus Subspace)\label{The:consensus}
Under Assumptions \ref{Ass:graph} and \ref{Ass:function}, for every $\tau \in [0, \frac{1}{2}-\tau_{\beta})$, with Condition (a) and (b),
\begin{align*}
\mathbb{P}\Big(\lim_{k\rightarrow \infty}(k+1)^{\tau}\big\|s_{i}^{k}-\bar{x}^{k}\big\| =0\Big)=1,\quad \forall i,
\end{align*}
where $\bar{x}^{k}=\frac{1}{n}\sum_{i=1}^{n}x_{i}^{k}$ is the network-averaged process.
\end{theorem}
\begin{proof}
With Assumption \ref{Ass:graph}, $\big(\bm{1}_{n} \otimes I_{d}\big)^{\top} \big( L^{k}\otimes I_{d}\big) = 0$. By \eqref{Alg1-1},
\begin{align}\label{A-1}
\bar{s}^{k}=&\bar{v}^{k}.
\end{align}
Define $\bm{\tilde{s}^{k}}=\bm{{s}^{k}}-\bm{1}_{n}\otimes \bar{x}^{k}$ and $\bm{\tilde{x}^{k}}=\bm{{x}^{k}}-\bm{1}_{n}\otimes \bar{x}^{k}$, for all $k\geqslant 0$. Since $\bm{\tilde{x}^k}\in \Lambda^{\perp}$, where $\Lambda^{\perp}$ is the orthogonal subspace of the consensus subspace $\Lambda$ and $\mathcal{P}_{nd}=\frac{1}{n}\big(\bm{1}_{n} \otimes I_{d}\big)\big(\bm{1}_{n} \otimes I_{d}\big)^{\top}$.
By \eqref{Alg1-1} and \eqref{A-1}, we have
\begin{align}\label{A-3}
\bm{\tilde{s}^{k}}=&\Big(I_{nd}-\beta^{k}\big(L^{k}\otimes I_{d}\big)\Big)\Big(\bm{{v}^{k}}-\bm{1}_{n}\otimes \bar{x}^{k}\Big)\notag\\
=&\Big(I_{nd}-\beta^{k}\big(L^{k}\otimes I_{d}\big)\Big)\Big(\bm{{s}^{k-1}}-\bm{1}_{n}\otimes \bar{x}^{k}\Big)
\notag\\&+\Big(I_{nd}-\beta^{k}\big(L^{k}\otimes I_{d}\big)\Big)\Big(\bm{x^{k}}-\bm{x^{k-1}}\Big)\notag\\
=&\Big(I_{nd}-\beta^{k}\big(L^{k}\otimes I_{d}\big)\Big)\bm{\tilde{s}^{k-1}}
\notag\\&+\Big(I_{nd}-\beta^{k}\big(L^{k}\otimes I_{d}\big)\Big)\Big(\bm{1}_{n}\otimes \bar{x}^{k-1}-\bm{1}_{n}\otimes \bar{x}^{k}\Big)
\notag\\&+\Big(I_{nd}-\beta^{k}\big(L^{k}\otimes I_{d}\big)\Big)\Big(\bm{x^{k}}-\bm{x^{k-1}}\Big)\notag\\
=&\Big(I_{nd}-\beta^{k}\big(L^{k}\otimes I_{d}\big)\Big)\bm{\tilde{s}^{k-1}}\notag\\&+\Big(I_{nd}-\beta^{k}\big(L^{k}\otimes I_{d}\big)\Big)\Big(\bm{\tilde{x}^{k}}-\bm{\tilde{x}^{k-1}}\Big),
\end{align}
for all $k\geqslant 0$.  Then, let us estimate $\big(\bm{\tilde{x}^{k}}-\bm{\tilde{x}^{k-1}}\big)$. From relation \eqref{Alg1-2}, we see that for all $k\geqslant 1$,
\begin{align}\label{A-4}
&\bm{\tilde{x}^{k}}-\bm{\tilde{x}^{k-1}}
\notag\\=&\bm{x^{k-1}}-\alpha^{k-1}\Big[ \nabla_{\bm{x^{k-1}}} g\big(\bm{x^{k-1}},\bm{s^{k-1}}\big)+\bm{\varsigma^{k-1}} \Big]+\gamma^{k-1}\bm{\iota^{k-1}}
\notag\\&+\Big(\bm{1}_{n}\otimes \bar{x}^{k-1}-\bm{1}_{n}\otimes \bar{x}^{k}\Big)-\bm{x^{k-1}}\notag\\
=&-\alpha^{k-1}\Big[ \nabla_{\bm{x^{k-1}}}g\big(\bm{x^{k-1}},\bm{s^{k-1}}\big)+\bm{\varsigma^{k-1}} \Big]+\gamma^{k-1}\bm{\iota^{k-1}}
\notag\\&+\alpha^{k-1}\bigg[\bm{1}_{n}\otimes \Big(\frac{1}{n}\sum_{i=1}^{n}\nabla_{x_{i}^{k-1}} g_{i}\big(x_{i}^{k-1},s_{i}^{k-1}\big)\Big)+\bm{1}_{n}\otimes\bar{\varsigma}^{k-1} \bigg]\notag\\
&-\gamma^{k-1}\Big(\bm{1}_{n}\otimes\bar{\iota}^{k-1}\Big )\notag\\
=&-\alpha^{k-1}T_{1}-\alpha^{k-1}T_{2}+\gamma^{k-1}T_{3},
\end{align}
where
\begin{align*}
\begin{cases}
T_{1}=&-\bm{1}_{n}\otimes \Big(\sum_{i=1}^{n}\frac{1}{n}\nabla_{x_{i}^{k-1}} g_{i}\big(x_{i}^{k-1},s_{i}^{k-1}\big)\Big)\notag\\&+\nabla_{\bm{{x}^{k-1}}} g\big(\bm{x^{k-1}},\bm{s^{k-1}}\big),\\
T_{2}=&\bm{\varsigma^{k-1}}-\bm{1}_{n}\otimes\bar{\varsigma}^{k-1},\\
T_{3}=&\bm{\iota^{k-1}}-\bm{1}_{n}\otimes\bar{\iota}^{k-1},
\end{cases}
\end{align*}
for all $k\geqslant 0$. Consider the $j$-th component of $T_{1}$:
\begin{align}\label{A-5}
T_{1}^{j}\doteq \nabla_{x_{j}^{k-1}} g_{j}\big(x_{j}^{k-1},s_{j}^{k-1}\big)-\frac{1}{n}\sum_{i=1}^{n}\nabla_{x_{i}^{k-1}} g_{i}\big(x_{i}^{k-1},s_{i}^{k-1}\big),
\end{align}
and note that $T_{1}^{j}$ may be composed as
\begin{align}\label{A-6}
&T_{1}^{j}\notag\\=
&\Big(\nabla_{x_{j}^{k-1}} g_{j}\big(x_{j}^{k-1},s_{j}^{k-1}\big)-\nabla_{\bar{x}^{k-1}} g_{j}\big(\bar{x}^{k-1},s_{j}^{k-1}\big)\Big)\notag\\
&+\Big(\nabla_{\bar{x}^{k-1}}g_{j}\big(\bar{x}^{k-1},s_{j}^{k-1}\big)-\frac{1}{n}\sum_{i=1}^{n}\nabla_{\bar{x}^{k-1}} g_{i}\big(\bar{x}^{k-1},s_{i}^{k-1}\big)\Big)\notag\\
&+\Big(\frac{1}{n}\sum_{i=1}^{n}\nabla_{\bar{x}^{k-1}} g_{i}\big(\bar{x}^{k-1},s_{i}^{k-1}\big)\notag\\
&-\frac{1}{n}\sum_{i=1}^{n}\nabla_{x_i^{k-1}} g_{i}\big(x_{i}^{k-1},s_{i}^{k-1}\big)\Big).
\end{align}
For the second term on the R.H.S of \eqref{A-6}, note that, by Assumption \ref{Ass:function}, there exists a constant $c_{1}>0$ such that
\begin{align}\label{A-7}
&\Big\|\nabla_{\bar{x}^{k-1}} g_{j}\big(\bar{x}^{k-1},s_{j}^{k-1}\big)-\frac{1}{n}\sum_{i=1}^{n}\nabla_{\bar{x}^{k-1}} g_{i}\big(\bar{x}^{k-1},s_{i}^{k-1}\big)\Big\|
\notag\\=&\Big\|\nabla_{\bar{x}^{k-1}} g_{j}\big(\bar{x}^{k-1},s_{j}^{k-1}\big)-\nabla_{\bar{x}^{k-1}} G\big(\bar{x}^{k-1},s_{i}^{k-1}\big)\Big\|\leqslant c_{1}.
\end{align}
By the Lipschitz continuity of gradients in Assumption \ref{Ass:function}, we have, for a sufficiently large $c_{2}$, we have
\begin{align}\label{A-8}
\big\|T_{1}^{j}\big\|\leqslant c_{1}+c_{2}\sum_{i=1}^{n}\big\| x_{i}^{k-1}-\bar{x}^{k-1}\big\|.
\end{align}
Hence, there exist constants $c_{3}$ and $c_{4}$ such that
\begin{align}\label{A-9}
\big\|T_{1}\big\|\leqslant c_{3}+c_{4}\big\|\bm{x^{k}}-\bm{1}_{n}\otimes \bar{x}^{k}\big\|=c_{3}+c_{4}\big\|\bm{\tilde{x}^{k-1}}\big\|.
\end{align}
For term $T_{2}$ in \eqref{A-4}, consider an arbitrarily small $\delta\in (0,\frac{1}{2})$. Define $R^{k}=(k+1)^{-\frac{1}{2}-\delta}\big\|\bm{\varsigma^{k}}-\bm{1}_{n}\otimes{\bar{\varsigma}^{k}}\big\|$ for all $k$. By Lemma \ref{Lem:martingale}, we have $\lim_{k\rightarrow\infty}R^{k}=0$ a.s. Since $\frac{1}{k}\leqslant \frac{2}{k+1}$ for all $k\geqslant1$, we have
\begin{align}\label{A-10}
\big\|\alpha^{k-1}T_{2}\big\|=\alpha^{k-1}k^{\frac{1}{2}+\delta}R^{k-1}\leqslant \frac{2c_{\alpha}}{k^{\frac{1}{2}-\delta}}R^{k-1}, \; \mbox{for large}\; k.
\end{align}
Similarly,
\begin{align}\label{A-11}
\big\|\gamma^{k-1}T_{3}\big\|\leqslant \frac{2c_{\gamma}\big\|T_{3}\big\|}{k^{\frac{1}{2}}\sqrt{\log\log(k-1)}}\leqslant \frac{2c_{\gamma}}{k^{\frac{1}{2}}}\big\|T_{3}\big\|, \; \mbox{for large}\; k.
\end{align}
Since $\big\|T_{3}\big\|$ has moments of all orders, by \eqref{A-10}-\eqref{A-11},  there exists $\mathcal{R}_{+}$-valued $\big\{ \mathcal{H}^{k}\big\}$-adapted process $\{M^{k}\}$ and $\{N^{k}\}$ such that
\begin{align}\label{A-12}
\big\|\alpha^{k-1}T_{2}\big\|+\big\|\gamma^{k-1}T_{3}\big\|\leqslant \frac{1}{k^{\frac{1}{2}-\delta}}M^{k}(1+N^{k}),\; \mbox{for large}\; k,
\end{align}
with $\{M^{k}\}$ being bounded a.s. and  $\{N^{k}\}$ possessing moments of all orders.
Since $\bm{\tilde{x}^{k}}\in \Lambda^{\perp}$ for all $k\geqslant 0$, by Lemma \ref{Lem:graph} there exists a $\{\mathcal{H}^{k+1}\}$ adapted $\mathcal{R}_{+}$-valued process $\{r^{k}\}$ and a constant $c_{5}>0$ such that $0\leqslant r^k\leqslant 1$ a.s. and
\begin{align}\label{A-13}
\Big\|\big( I_{nd}-\beta^k\big( L^{k}\otimes I_{d}\big)\big)\bm{\tilde{x}^{k}}\Big\|\leqslant (1-r^{k})\Big\|\bm{\tilde{x}^{k}}\Big\|,
\end{align}
with
\begin{align}\label{A-14}
\mathbb{E}\big[r^{k}|\mathcal{H}^{k}\big]\geqslant \frac{c_{5}}{(k+1)^{\tau_{\beta}}}, \quad \mbox{a.s.}
\end{align}
for all $k$ large enough.
Thus, by \eqref{A-3}, \eqref{A-4}, \eqref{A-9}, \eqref{A-12} and \eqref{A-13} we obtain that for large $k$,
\begin{align}\label{A-15}
&\big\|\bm{\tilde{s}^{k}}\big\|\notag\\
\leqslant&(1-r^{k})\big\|\bm{\tilde{s}^{k-1}}\big\|+(1-r^{k})\big\|\bm{\tilde{x}^{k}}-\bm{\tilde{x}^{k-1}}\big\|\notag\\
\leqslant&\big( 1-r^{k}\big)\big\|\bm{\tilde{s}^{k-1}}\big\|+ \big( 1-r^{k}\big)\alpha^{k-1}c_{3}\notag\\
&+(1-r^{k})\alpha^{k-1}c_{4}\big\|\bm{\tilde{x}^{k-1}}\big\|+\frac{1-r^{k}}{k^{\frac{1}{2}-\delta}}M^{k}(1+N^{k}).
\end{align}
Since $\alpha^{k}=\frac{c_{\alpha}}{k}\leqslant \frac{2}{(k+1)^{\frac{1}{2}-\delta}}$, by \eqref{A-15},  we have
\begin{align}\label{A-16}
\big\|\bm{\tilde{s}^{k}}\big\|\leqslant \big( 1-r^{k}\big)\big\|\bm{\tilde{s}^{k-1}}\big\|+\frac{ c_{7}}{k^{\frac{1}{2}-\delta}}M^{k}(1+N^{k})\
\end{align}
for large $k$ and a constant $c_{7}$. According to Lemma \ref{Lem:martingale}, we conclude that for all $\tau$ and $\epsilon_{1}>0$ with
\begin{align}\label{A-17}
0\leqslant \tau<\frac{1}{2}-\delta-\tau_{\beta}-\frac{1}{2+\epsilon_{1}},
\end{align}
we have $\lim_{k\rightarrow \infty}(k+1)^{\tau}\bm{\tilde{s}^{k}}=0$. By taking $\epsilon_{1}\rightarrow \infty$ (Since $N^{k}$ possesses moments of all orders) and $\delta\rightarrow 0$, $\lim_{k\rightarrow \infty}(k+1)^{\tau}\bm{\tilde{s}^{k}}=0$ for all $\tau\in [0,\frac{1}{2}-\tau_{\beta})$, which completes the proof.
\end{proof}
Theorem  \ref{The:consensus} shows that variables $s_{i}^{k}$ can track $\bar{x}^{k}$ almost surely, which is unavailable for agent $i$ .

\subsection{Weak convergence}

In this section,  the weak convergence of the agent estimates $\{x_{i}^k\}$ to the set of global minima of $g(\cdot,\cdot)$ is given. The following assumption for common function $G(x,y)$ given in Assumption \ref{Ass:function} is required.

\begin{assumption}\label{Ass:common function}
$G:\mathcal{R}^{d}\times \mathcal{R}^{d} \rightarrow \mathcal{R}_+$ is a twice differentiable function such that
\begin{itemize}
\item[(a)] $\min_{x} G(x,y)=0$.
\item[(b)] $\lim_{\|x\|\rightarrow\infty} G(x,y)=\infty$ and $\lim_{\|x\|\rightarrow\infty} \|\partial_x G(x,y)\|=\infty$.
\item[(c)] $\inf \big(\|\partial_{x} G(x,y)\|^{2}-\bigtriangleup_{x} G(x,y) \big)>-\infty$.
\item[(d)]For $\epsilon>0$, let $d\pi^{\epsilon}(x)=\frac{1}{Z^{\epsilon}}\exp\big(-\frac{2G(x,y)}{\epsilon^{2}})\big)dx$, and
 $X^{\epsilon}=\int \exp\big(\frac{-2G(x,y)}{\epsilon^{2}}\big)dx<\infty$. $G$ satisfies that $\pi^{\epsilon}$ has a weak limit $\pi$ as $\epsilon\rightarrow 0$.
\item[(e)]$\lim\inf_{\|x\|\rightarrow \infty}\Big\langle \frac{\partial_x G(x,y)}{\|\partial_x G(x,y)\|},\frac{x}{\|x\|}\Big\rangle\geqslant C(d)$, where $C(d)=\Big(\frac{4d-4}{4d-3}\Big)^{\frac{1}{2}}$.
\item[(f)]$\lim\inf_{\|x\|\rightarrow \infty}\frac{\|\partial_x G(x,y)\|}{\|x\|}>0$.
\item[(g)]$\lim\sup_{\|x\|\rightarrow \infty}\frac{\|\partial_x G(x,y)\|}{\|x\|}<\infty$.
\end{itemize}
\end{assumption}

\begin{remark}
Assumption \ref{Ass:common function} is a modification of Assumption \ref{Ass:global function} for seeking social optimum of aggregative games from centralized annealing assumptions \cite{gelfand1991recursive} given in Appendix.
\end{remark}
We now state the weak convergence of the agent estimates $\{x_{i}^k\}$ to the set of global optimum of $g(\cdot,\cdot)$.

\begin{theorem}\label{The:Convergence}
Under Assumptions \ref{Ass:graph}- \ref{Ass:common function},  for any bounded contimuous function $f_{i}:\mathcal{R}^{d}\times\mathcal{R}^{d}\rightarrow \mathcal{R}$,
\begin{align*}
\lim_{k\rightarrow\infty}\mathbb{E}_{0,x_{i}^{k}}[f_{i}(x_{i}^{k},s_{i}^{k})]=\pi\Big(f_{i}\Big).
\end{align*}
\end{theorem}

\begin{proof}
According to \eqref{Alg1-2}, we have
\begin{align}\label{B-1}
x_{i}^{k+1}=&x_{i}^{k}-\alpha^{k}\Big[ \nabla_{x_{i}^{k}} g_{i}\big(x_{i}^{k},s_{i}^{k}\big)+\varsigma_{i}^{k} \Big]+\gamma^{k}\iota_{i}^{k}\notag\\
=&x_{i}^{k}-\alpha^{k}\Big[ \nabla_{x_{i}^{k}} g_{i}\big(x_{i}^{k},\bar{x}^{k}\big)+\nabla_{x_{i}^{k}} g_{i}\big(x_{i}^{k},s_{i}^{k}\big)
\notag\\&-\nabla_{x_{i}^{k}} g_{i}\big(x_{i}^{k},\bar{x}^{k}\big)+\varsigma_{i}^{k} \Big]+\gamma^{k}\iota_{i}^{k}.
\end{align}
Fix $\tau\in[0,\frac{1}{2}-\tau_{\beta})$  with any $\delta>0$. By Theorem \ref{The:consensus},
\begin{align}\label{B-2}
\lim_{k\rightarrow\infty } k^{\tau}\big\| s_{i}^k-\bar{x}^{k}\big\|=0
\end{align}
holds for all $i$. According to Egorov's theorem, there exists a constant $D_{\delta}>0$ such that
\begin{align}\label{B-3}
\mathbb{P}\Big(\sup_{k\rightarrow\infty} k^{\tau}\big\| s_{i}^k-\bar{x}^{k} \big\|\leqslant D_{\delta}\Big)>1-\delta, \quad \forall i.
\end{align}
By Assumption \ref{Ass:function},
\begin{align}\label{B-4}
&\mathbb{P}\Big( \sup_{k\geqslant 0} k^{\tau}\|s_{i}^{k}-\bar{x}^{k}\|\leqslant D_{\delta}\Big)\notag\\\leqslant &\mathbb{P}\Big( \sup_{k\geqslant 0} k^{\tau}\big\| \nabla_{x_{i}^{k}} g_{i}\big(x_{i}^{k},s_{i}^{k}\big)-\nabla_{x_{i}^{k}} g_{i}\big(x_{i}^{k},\bar{x}^{k}\big)\big\|\leqslant LD_{\delta}\Big).
\end{align}
Define $R_{i}^{k}=\nabla_{x_{i}^{k}} g_{i}\big(x_{i}^{k},s_{i}^{k}\big)-\nabla_{x_{i}^{k}} g_{i}\big(x_{i}^{k},\bar{x}^{k}\big)$.
Consider the $\mathcal{H}^{k}$-process $\{ R_{i\delta}^{k}\}$, given by
\begin{align}\label{B-5}
R_{i\delta}^{k}=\begin{cases}
R_{i}^{k}\;\;, \quad \mbox{if}\quad k^{\tau}\|R_{i}^{k}\|\leqslant LD_{\delta}\\
\frac{LD_{\delta}}{k^{\tau}},\quad \mbox{if}\quad k^{\tau}\|R_{i}^{k}\|> LD_{\delta},
\end{cases}
\end{align}
for all $k\geqslant 0$. By construction, we have
\begin{align}\label{B-6}
\mathbb{P}\Big( \sup_{k\geqslant 0}k^{\tau}\big\| R_{i\delta}^{k}\big\|\leqslant LD_{\delta}\Big)=1.
\end{align}
Consider the stochastic process $\{x_{i\delta}^{k}\}$, which evolves as:
\begin{align}\label{B-7}
x_{i\delta}^{k+1}=x_{i\delta}^{k}-\alpha^{k}\Big[ \nabla_{x_{i\delta}^{k}} g_{i}\big(x_{i\delta}^{k},\bar{x}_{\delta}^{k}\big)+R_{i\delta}^{k}+\varsigma_{i}^{k} \Big]+\gamma^{k}\iota_{i}^{k}.
\end{align}
with initial condition $x_{i\delta}^{0}=x_{i}^{0}$.
$\{x_{i\delta}^{k}\}$ is $\mathcal{H}^{k}$-adapted and the process $\{x_{i\delta}^{k}\}$ and $ \{x_{i}^{k}\}$ agree on events $\big\{\sup_{k\geqslant 0}k^{\tau}\big\| R_{i}^{k}\big\|\leqslant LD_{\delta} \big\}$, so that
\begin{align}\label{B-8}
\mathbb{P}\Big( \sup_{k\geqslant 0} \big\|  x_{i\delta}^{k}- x_{i}^{k}\big\|>0\Big)\leqslant \delta.
\end{align}
Define $\xi_{i}^{k}=R_{i\delta}^{k}+\varsigma_{i}^{k} $ for all $k\geqslant 0$ and denote by $\mathcal{F}_{\delta}^{k}$ the $\sigma$-algebra:
\begin{align}\label{B-9}
\mathcal{F}_{\delta}^{k}=\sigma\big(\bm{x_{\delta}^{0}},\bm{v_{\delta}^{0}}, L^{0},\ldots, L^{k-1},\bm{\varsigma^{0}},\ldots,\bm{\varsigma^{k-1}},\bm{ \iota^{0}},\ldots, \bm{ \iota^{k-1}}\big).
\end{align}
For all $k\geqslant 0$, $\mathcal{F}_{\delta}^{k}\subset \mathcal{H}^{k}$ holds .
By Condition (a) and \eqref{B-6},
\begin{align}\label{B-10}
\Big\| \mathbb{E}\big[(R_{i\delta}^{k}+\varsigma_{i}^{k})|\mathcal{F}_{\delta}^{k}\big]\Big\|\leqslant \Big\| \mathbb{E}\big[R_{i\delta}^{k}|\mathcal{F}_{\delta}^{k}\big]\Big\|\leqslant \frac{LD_{\delta}}{k^\tau}
\end{align}
holds almost surely, and by the parallelogram law,
\begin{align}\label{B-11}
&\mathbb{E}\Big[\big\|R_{i\delta}^{k}+\varsigma_{i}^{k}\big\|^{2}\big |\mathcal{F}_{\delta}^{k} \Big]\notag\\
\leqslant & 2\mathbb{E}\Big[\big\|R_{i\delta}^{k}\big\|^{2}\big |\mathcal{F}_{\delta}^{k} \Big]+2\mathbb{E}\Big[\big\|\varsigma_{i}^{k}\big\|^{2}\big |\mathcal{F}_{\delta}^{k}\Big]
\leqslant 2C+\frac{2L^{2}D_{\delta}^{2}}{k^{2\tau}}.
\end{align}
Therefore, for any $i$, the process $\{x_{i\delta}^{k}\}$ falls under purview of Lemma \ref{Lem:global convergence}. Specially, taking $\nu_{1}=0$, $\nu_{2}=\tau$ and letting $\mathcal{I}^{k}=\mathcal{F}_{\delta}^{k}$, Assumption \ref{Ass:global noise} in Appendix is satisfied. Therefore,
\begin{align}\label{B-12}
\lim_{k\rightarrow\infty}\mathbb{E}\big[ f_{i}(x_{i\delta}^{k},s_{i}^{k})\big]=\pi(f_{i}),
\end{align}
according to Lemma \ref{Lem:global convergence}.
By \eqref{B-8},
\begin{align}\label{B-13}
&\big\|\mathbb{E}[ f_{i}(x_{i}^{k},s_{i}^{k})]-\pi(f_{i})\big\|\notag\\
\leqslant&\mathbb{E}\big[\|f_{i}(x_{i}^{k},s_{i}^{k})-f_{i}(x_{i\delta}^{k},s_{i}^{k})\|\big]+\big\|\mathbb{E}[ f_{i}(x_{i\delta}^{k},s_{i}^{k})]-\pi(f_{i})\big\|\notag\\
\leqslant & 2\|f\|_{\infty}\delta+\big\|\mathbb{E}[ f_{i}(x_{i\delta}^{k},s_{i}^{k})]-\pi(f_{i})\big\|.
\end{align}
By \eqref{B-12}, we have
\begin{align}\label{B-14}
\lim\sup_{k\rightarrow\infty}\big\|\mathbb{E}\big[ f_{i}(x_{i}^{k},s_{i}^{k})\big]-\pi(f_{i})\big\|\leqslant 2\|f\|_{\infty}\delta
\end{align}
holds for any $\delta>0$.
Therefore,
\begin{align}\label{B-15}
\lim_{k\rightarrow\infty}\Big\|\mathbb{E}\big[ f_{i}(x_{i}^{k},s_{i}^{k})\big]-\pi(f_{i})\Big\|=0,
\end{align}
which completes the proof.
\end{proof}
Theorem \ref{The:Convergence} shows the weak convergence of Algorithm \ref{Alg1} to the social optimum of aggregative games.

\section{Simulation}\label{sec5}

In this section, we provide an example of flexible electric vehicle charging control \cite{liu2020approximate}, whose convex version is studied in \cite{jacquot2017demand}.

To be specific, consider there are many residents commute by cars every day in a neighborhood and the cost is different for each resident due to his  occupation with respect to hours of the day. Player $i$'s electricity bill is defined by
 $$b_{i}(x_{i})=\frac{a_i}{1+\exp{-(x_i-b_i)}}+c_i \log(1+(x_i-d_i)^2),$$
which is a non-convex function with respect to hours $x_{i}$ over the day. $a_i$ and $c_i$ are independently and uniformly distributed random variables over $[5,40]$ and $d_i$ is given constant represent the optimal departure time for different resident $i$. Resident $i$'s cost is then defined by $$g_i(x_i,\bar{x})=b_{i}(x_{i})+\lambda_i(x_i-\bar{x})^2, $$
where $\lambda_i$ indicates his sensitivity to the deviation from public preference $\bar{x}$.
Specifically, we take $d=(7,7,8,8,9,9,13,19,19,22)^\top$, $b=(7,7.4,7.8,8.2,8.6,9,9.4,9.8,10.2,10.6)^{\top}$ and $\lambda_i$ as a random value in $(0,2)$ which indicates resident $i$'s sensitivity to the deviation from average departure time.

The communication topology between agents  is performed over an Erd\H{o}s-R\'{e}nyi random graph. Consider a graph set $\mathcal{G}$ containing 50 graphs, each of which is generated according to the E-R graph $G(10,p)$, where the probability $p$ is
selected independently and uniformly over $[0.1,0.2]$. At each iteration, a graph is randomly selected from the graph set $\mathcal{G}$.

 Therefore, the social optimum seeking problem is given as follows:
\begin{align*}
    \min_{x_i}\; &G(x) = \sum_{i=1}^{10} g_i(x_i,\bar{x})\\
    & \bar{x}=\frac{x_1+\cdots+x_{10}}{10}
\end{align*}

Define the noise sequences $\{\varsigma^k_i\}_{k \geqslant 0}$ as independently and uniformly distributed random variables over $[-5,5]$, $\{\iota_i^k\}_{k \geqslant 0}$ as  i.i.d. random variables with Gaussian distribution $N(0,1)$, and the step size $\{\alpha^k\}_{k \geqslant 0}$,$\{\beta^k\}_{k \geqslant 0}$,$\{\gamma^k\}_{k \geqslant 0}$ as given in Algorithm \ref{Alg1}.

By randomly selecting the initial positions of $x_{i}^{k}$'s, performing the distributed annealing algorithm gives rise to evolutions
of all $s_{i}^{k}$, $\bar{x}^{k}$ and $x_{i}^{k}$'s in Figs. \ref{better_21}-\ref{better_33}. Figure \ref{better_21} shows the trajectories of $(k+1)^{\tau}(s_i^k-\bar{x}^k)$ of each resident which validates the effectiveness of Theorem \ref{The:consensus}. Fig. \ref{better_22} gives the evolutions of  $s_{i}^{k}$ and $\bar{x}^{k}$, showing that the evolutions $s_{i}^{k}$s converge to $\bar{x}$, which is the network-averaged process.  Figure \ref{better_23} and Figure \ref{better_33} show two different weak convergence results which $x_i^k$ converges stationarily. In addition, we compare our algorithm with DAAG given in \cite{koshal2016distributed}. The evolutions of $x_{i}^{k}$ are provided for DAAG in Fig. \ref{better_26} and the comparisons of $\sum g_{i}(x_{i},\bar{x})$ between our algorithm and DAAG are presented in Fig.\ref{better_27}. Fig.\ref{better_27} shows that $\sum g_{i}(x_{i},\bar{x})$ for our algorithm is much smaller that for DAAG, which provide weak convergence results for seeking  social optimum. Therefore, the
simulation results support the theoretical results.

\begin{figure}[htbp]
  \centering
  \includegraphics[width=1.0\linewidth]{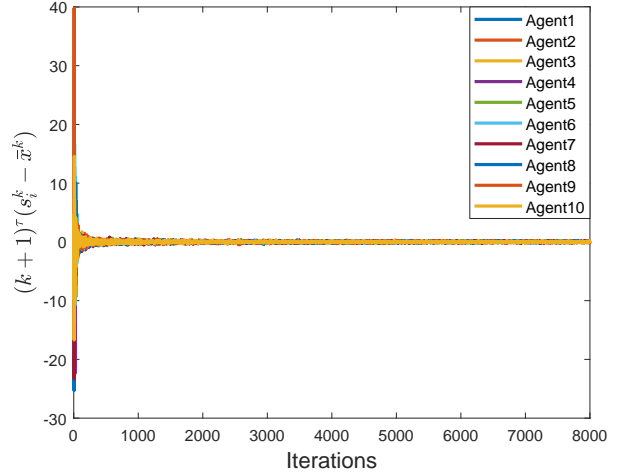}
  \caption{The trajectories of $(k+1)^{\tau}(s_i^k-\bar{x}^k)$ of each agent.}\label{better_21}
\end{figure}
\begin{figure}[htbp]
  \centering
  \includegraphics[width=1.0\linewidth]{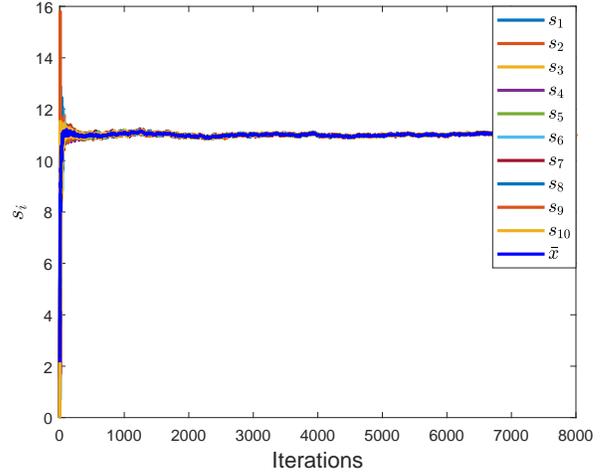}
  \caption{The trajectories of $\bar{x}^k$ and $s_i^k$ of each agent}\label{better_22}
\end{figure}
\begin{figure}[htbp]
  \centering
  \includegraphics[width=1.0\linewidth]{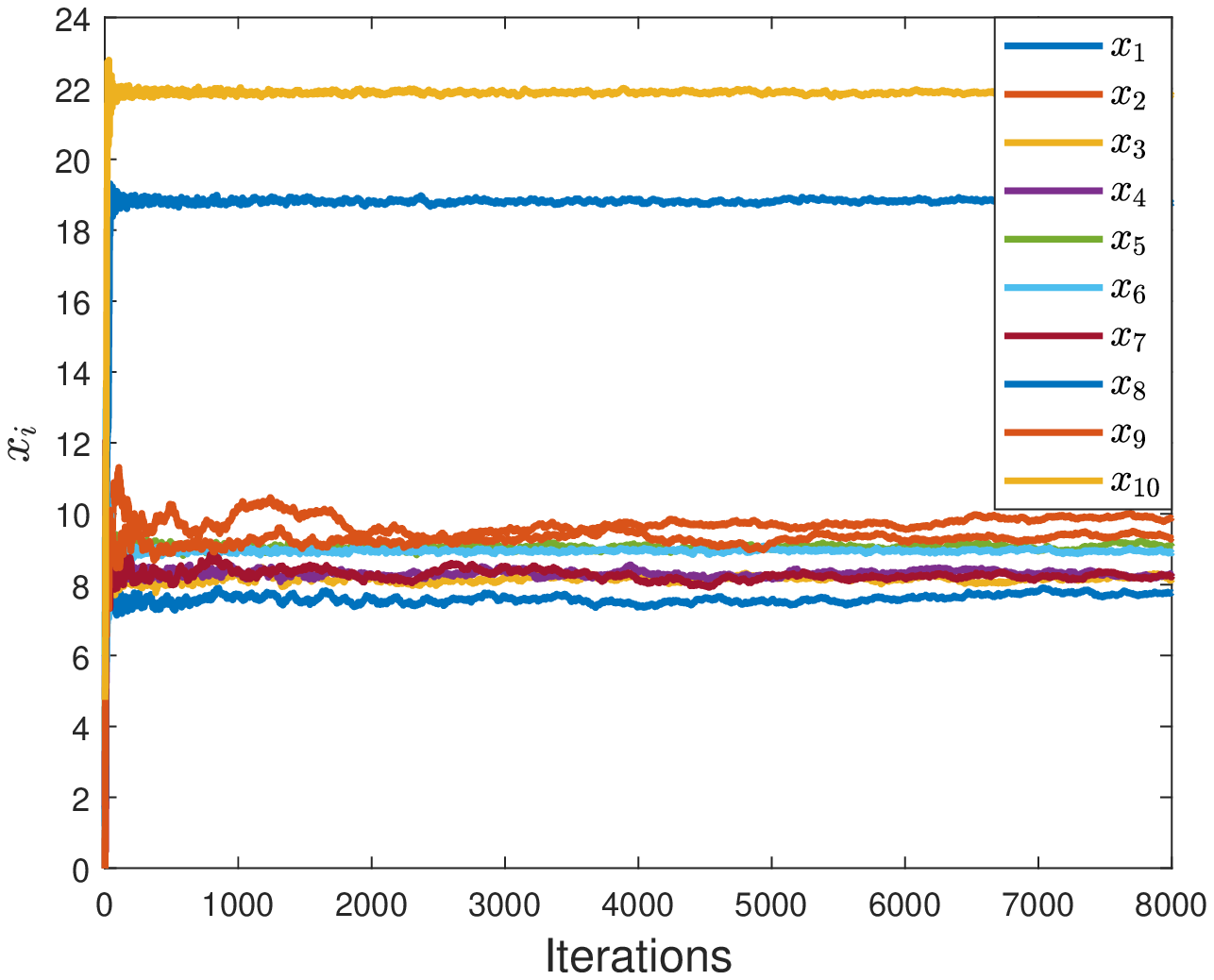}
  \caption{The trajectories of $x_i^k$ of each agent by DAA}\label{better_23}
\end{figure}
\begin{figure}[htbp]
  \centering
  \includegraphics[width=1.0\linewidth]{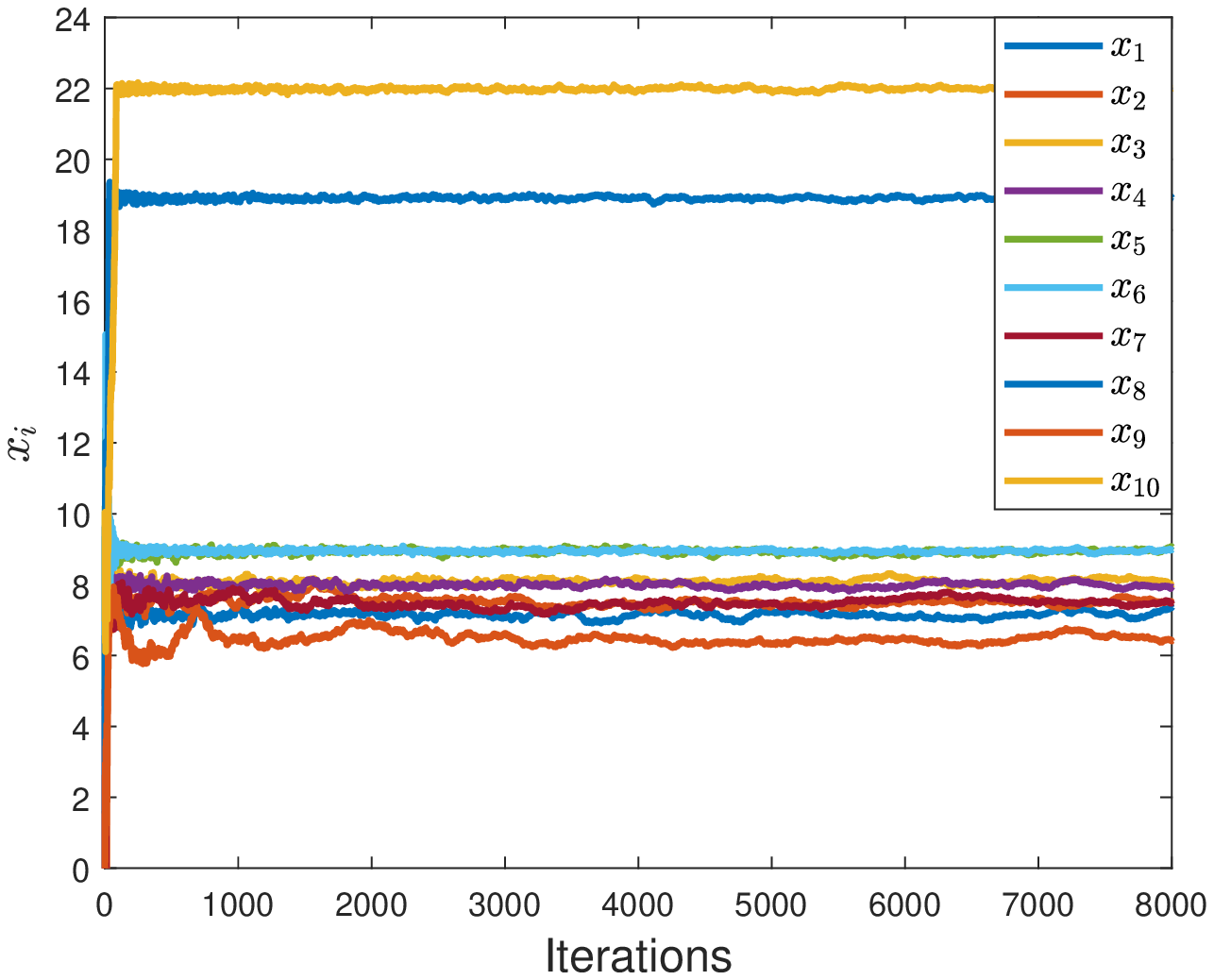}
  \caption{The trajectories of $x_i^k$ of each agent by DAA}\label{better_33}
\end{figure}
\begin{figure}[htbp]
  \centering
  \includegraphics[width=1.0\linewidth]{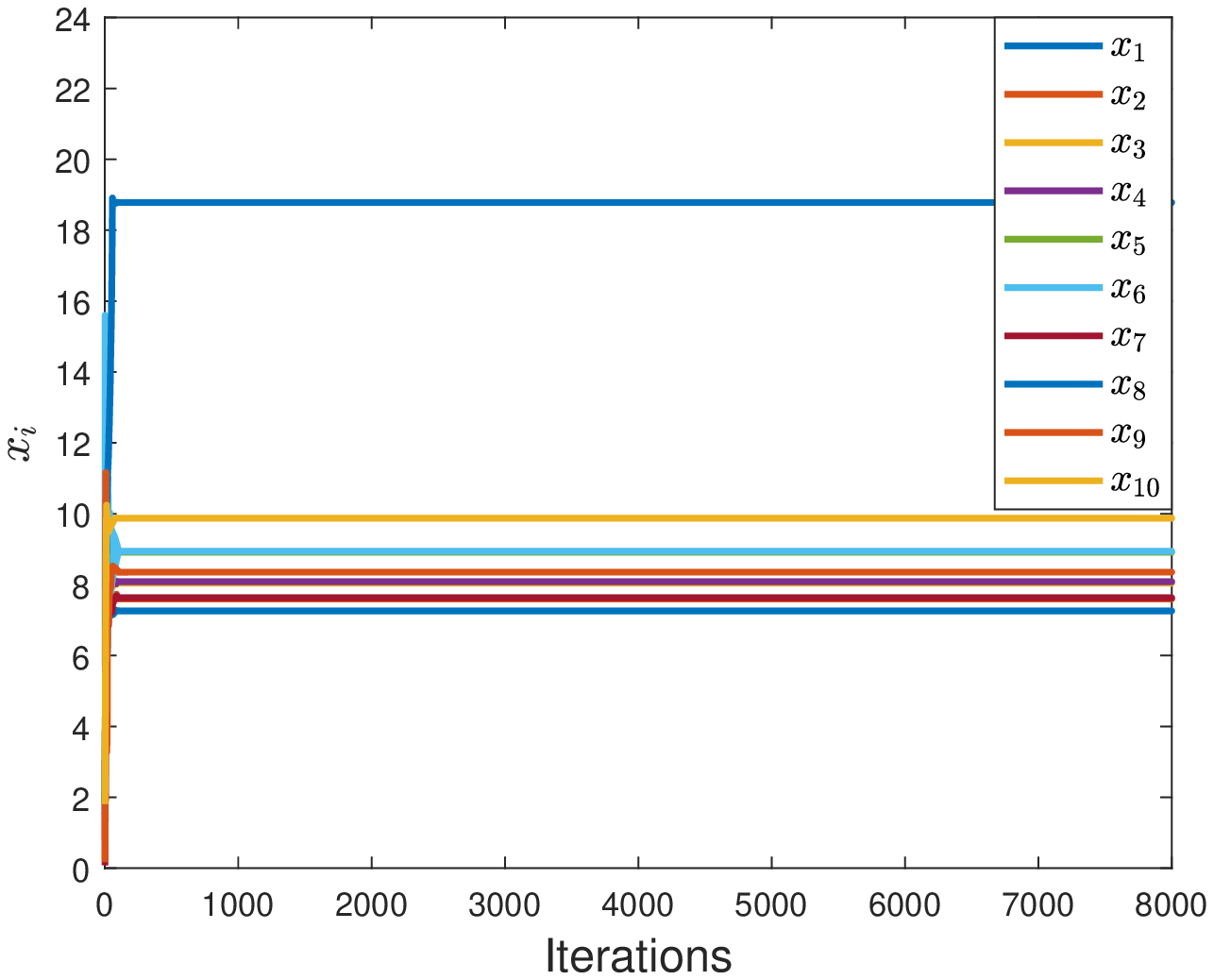}
  \caption{The trajectories of $x_i^k$ of each agent by DAAG}\label{better_26}
\end{figure}
\begin{figure}[htbp]
  \centering
  \includegraphics[width=1.0\linewidth]{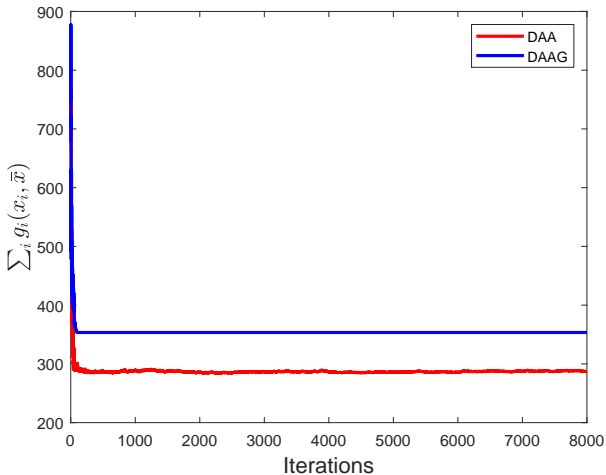}
  \caption{Changes in the values of the global objective function}\label{better_27}
\end{figure}
\section{Conclusion}\label{sec6}
In this paper,  the seeking of social optimum of cooperative aggregative games was studied. A distributed annealing algorithm was designed for seeking the social optimum.  Moreover, the weak convergence to the social optimum of the proposed algorithm was given. Finally, an example was given to the proposed algorithm to verify its effectiveness.

\appendices
\section{Centralized annealing Algorithm}
In the Appendix, we briefly review classical results \cite{gelfand1991recursive} that are used in the weak convergence proof.

Consider the following optimization problem:
\begin{align}
\min_{z\in \mathcal{R}^{d}} G(z)
\end{align}
where $G:\mathcal{R}^{d}\rightarrow \mathcal{R}_+$.
Construct the following stochastic recursion algorithm in $\mathcal{R}^{d}$:
\begin{align}\label{eq:global algorithm}
z^{k+1}=z^{k}-\alpha^{k}\Big(\nabla G\big(z^{k}\big)+\xi^{k}\Big)+\gamma^{k}w^{k},\quad k\geqslant 0,
\end{align}
where $G:\mathcal{R}^{d}\rightarrow \mathcal{R}_{+}$, $\{\xi^{k}\}$ is a sequence of $\mathcal{R}^{d}$-valued random variables, $\{w^{k}\}$ is a sequence of $\mathcal{R}^{d}$-valued i.i.d. Gaussian random variables with zero mean and covariance $I_{d}$. Further, we assume that
\begin{align*}
\alpha^{k}=\frac{c_{\alpha}}{k},  \;\mbox{and}\; \gamma^{k}=\frac{c_{\gamma}}{k^\frac{1}{2}\sqrt{\log\log k}}\quad \mbox{for large}\; k,
\end{align*}
where $c_{\alpha}$ and $c_{\gamma}$ are constants.
Next, consider the following assumptions on $G(\cdot)$, the gradient fields $\nabla G(\cdot)$ and noise $\{\xi^{k}\}$:

\begin{assumption}\label{Ass:global function}
$G:\mathcal{R}^{d}\rightarrow \mathcal{R}_+$ is a twice differentiable function such that
\begin{itemize}
\item[(a)] $\min\limits_z G(z)=0$.
\item[(b)] $\lim_{\|z\|\rightarrow\infty} G(z)=\infty$ and $\lim_{\|z\|\rightarrow\infty} \|\nabla G(z)\|=\infty$.
\item[(c)] $\inf \big(\|\nabla G(z)\|^{2}-\bigtriangleup G(z) \big)>-\infty$.
\item[(d)]For $\epsilon>0$, let
\begin{align*}
d\pi^{\epsilon}(z)&=\frac{1}{Z^{\epsilon}}\exp\Big(-\frac{2G(z)}{\epsilon^{2}}\Big)dz, \notag\\ Z^{\epsilon}\quad &=\int \exp\Big( \frac{-2G(z)}{\epsilon^{2}}\Big)dz.
\end{align*}
$G$ satisfies that $\pi^{\epsilon}$ has a weak limit $\pi$ as $\epsilon\rightarrow 0$.
\item[(e)]$\lim\inf_{\|z\|\rightarrow \infty}\Big\langle \frac{\nabla G(z)}{\|\nabla G(z)\|},\frac{z}{\|z\|}\Big\rangle\geqslant C(d)$, where $C(d)=\Big(\frac{4d-4}{4d-3}\Big)^{\frac{1}{2}}$.
\item[(f)]$\lim\inf_{\|z\|\rightarrow \infty}\frac{\|\nabla G(z)\|}{\|z\|}>0$.
\item[(g)]$\lim\sup_{\|z\|\rightarrow \infty}\frac{\|\nabla G(z)\|}{\|z\|}<\infty$.
\end{itemize}
\end{assumption}
Let $\mathcal{I}^{k}$ be a filtration generated by \eqref{eq:global algorithm}:
\begin{align}\label{eq:Global sigma algebra}
\mathcal{I}^{k}=\sigma\big( \{z^{0},\xi^{1},\ldots,\xi^{k-1},w^{1},\ldots,w^{k-1}\}\big).
\end{align}

\begin{assumption}\label{Ass:global noise}
There exists a constant $K_{1}>0$ such that
$\mathbb{E}\big(\big\|\xi^{k}\big\|^{2}\big| \mathcal{I}^{k}\big)\leqslant K_{1}\big(\alpha^{k})^{\nu_{1}}$, and
$\big\|\mathbb{E}\big(\xi^{k}\big |\mathcal{I}^{k}\big)\big\|\leqslant K_{1}\big(\alpha^{k}\big)^{\nu_{2}}$, a.s. with $\nu_{1}>-1$ and $\nu_{2}>0$.
\end{assumption}

The following results of Algorithm \eqref{eq:global algorithm} was obtained in \cite{gelfand1991recursive}:
\begin{lemma}\cite[Theorem 4]{gelfand1991recursive}\label{Lem:global convergence} Let Assumptions \ref{Ass:global function}-\ref{Ass:global noise} hold and assume $c_{\alpha}$ and $c_{\gamma}$ in \eqref{eq:global algorithm} satisfy
$\frac{c_{\gamma}^{2}}{c_{\alpha}}>K_{0}$. Then for any bounded continuous function $f:\mathcal{R}^{d}\rightarrow \mathcal{R}$, we have
\begin{align*}
\lim_{k\rightarrow\infty} \mathbb{E}_{0,z^{0}}\big[f(z^{k})\big]=\pi(f).
\end{align*}
\end{lemma}

\bibliographystyle{IEEEtran}
\bibliography{mybibfile}

\end{document}